\documentclass[fleqn,10pt]{IEEEtran}
\pdfoutput=1

\usepackage{graphicx}
\usepackage{amssymb}
\usepackage{amsmath}
\usepackage{amsthm}
\usepackage{epstopdf}
\usepackage{enumerate}
\usepackage{cite}

\newcommand{\thh}{\ensuremath{^\text{th}}}
\newcommand{\Z}{\ensuremath{\mathbb{Z}}}

\theoremstyle{definition}
\newtheorem{defn}{Definition}

\newtheorem{lem}{Lemma}
\newtheorem{remark}{Remark}

\newtheorem{prop}{Proposition}

\newtheorem{experiment}{Experiment}
\DeclareMathOperator{\std}{std}

\title{On the Convergence of the 
Iterative Shrinkage/Thresholding Algorithm With a Weakly Convex Penalty}
\author{\.Ilker Bayram 
 \thanks{ \.{I}. Bayram is with the Dept. of Electronics and Telecommunications Engineering, Istanbul Technical University, Istanbul, Turkey (e-mail~:~ibayram@itu.edu.tr). }
}
\date{}

\begin{document}
\maketitle
\begin{abstract}
We consider the iterative shrinkage/thresholding algorithm (ISTA) applied to a cost function composed of a data fidelity term and a penalty term. The penalty is non-convex but the concavity of the penalty is accounted for by the data fidelity term so that the overall cost function is convex. We provide a generalization of the convergence result for ISTA viewed as a forward-backward splitting algorithm.
We also demonstrate experimentally that for the current setup, using large stepsizes in ISTA can accelerate convergence more than existing schemes proposed for the convex case, like TwIST or FISTA.
\end{abstract}

\section{Introduction}

The iterative shrinkage/thresholding algorithm (ISTA) is widely used for solving  minimization problems of the form
\begin{equation}\label{eqn:C}
\min_x\,\Bigl\{C(x) = \frac{1}{2}\,\|y - H\,x\|_2^2 + P(x) \Bigr\}
\end{equation}
where $y$ is an observed signal, $H$ is a linear operator, and $P(x)$ is a penalty term reflecting our prior knowledge about the object to be recovered \cite{dau04p413,fig03p906,DeMol}. 
For minimizing $C(x)$, ISTA employs iterations of the form
\begin{equation}\label{eqn:ISTA}
x^{k+1} = T_{\alpha}\,\Bigl( x^k + \alpha\,H^T\,(y - H\,x^k)\Bigr),
\end{equation}
where $T_{\alpha}$ is the shrinkage/thresholding operator associated with $P$, defined as,
\begin{equation}\label{eqn:T}
T_{\alpha}(z) = \arg \min_x\, \frac{1}{2\alpha}\| x - z\|_2^2 + P(x).
\end{equation}
When $P(x)$ is convex, $T$ is also referred to as the proximity operator of $P$ \cite{combettes_chp}. We remark that if the minimization problem in \eqref{eqn:T} is strictly convex, then it has a unique solution. In that case, $T_{\alpha}$ will be well-defined. In the following, we denote the least and greatest eigenvalue of $H^T\,H$ with $\sigma_m$ and $\sigma_M$ respectively. We also assume throughout the paper that the functions of interest are proper, lower semi-continuous and $C(x)$ has a non-empty set of minimizers. 

For convex $P$, ISTA can be derived through different approaches. The majorization-minimization (MM) framework \cite{fig07p980,hun04p30, Lange}, enforces the constraint $\alpha < 1/\sigma_M$, and implies that the iterates achieve monotone descent in the cost. The descent property achieved by MM does not directly imply that the algorithm converges. However, the following proposition can be shown by making use of this descent property \cite{byr14p681} and interpreting the algorithm as an instance of a sequential unconstrained minimization algorithm (SUMMA) \cite{ByrneItOpt}.
\begin{prop}\label{prop:MM0}
Suppose $P(x)$ is convex, proper, lower semi-continuous and the set of minimizers of $C(x)$ is non-empty. If ${0< \alpha \leq 1/\sigma_M}$, then the sequence $x^k$ in \eqref{eqn:ISTA} monotonically reduces the cost, i.e., satisfies $C(x^{k+1}) \leq C(x^k)$, and converges to a minimizer of $C(x)$. \qed
\end{prop}

The forward-backward splitting algorithm \cite{com05p168} gives exactly the same iterations as \eqref{eqn:ISTA}, but allows $\alpha < 2/\sigma_M$. In practice, larger step-sizes accelerate  convergence to a minimizer. Although the MM interpretation is not valid when $ \alpha > 1/\sigma_M$, it can be shown also for ${1/\sigma_M < \alpha < 2/\sigma_M}$ that the cost decreases monotonically with each iteration \cite{she09p384}. However, the convergence proofs in \cite{com05p168,Bauschke} do not make use of this monotone descent property. Rather, the main object of interest in \cite{com05p168,Bauschke} is the distance to the set of minimizers, which shows a monotone behavior.
\begin{prop}\cite{com05p168,Bauschke} \label{prop:FB}
Suppose $P(x)$ is convex, proper, lower semi-continuous and the set of minimizers of $C(x)$ is non-empty. If ${0< \alpha < 2/\sigma_M}$, then the iterates $x^k$ in \eqref{eqn:ISTA} converge to a minimizer of $C(x)$. \qed
\end{prop}

When $P(x)$ is non-convex, ISTA is still applicable but convergence to a global minimizer is not guaranteed in general. Further, the convergence proof for the forward-backward scheme in \cite{com05p168,Bauschke} is not directly valid. In this paper, instead of an arbitrary non-convex penalty, we consider weakly-convex penalties  \cite{via83p231}\footnote{See specifically Defn.1 and Prop.~4.3 in \cite{via83p231}. We use a slightly different definition than \cite{via83p231}, in order to simplify notation. Our $\rho$-weakly convex functions correspond to $(-\rho/2)$-convex functions of \cite{via83p231}. }. 
\begin{defn} \label{defn:weaklyconvex}
A function $f:\mathbb{R}^n \rightarrow \mathbb{R}$ is said to be $\rho$-weakly convex  if 
\begin{equation*}
h(x) = \frac{s}{2} \, \|x\|_2^2 + f(x)
\end{equation*}
is convex when $ s\geq \rho \geq 0$. \qed
\end{defn}
If $P(x)$ is $\rho$-weakly convex and $0 < \rho \leq \sigma_m$, then $C(x)$ can be shown to be convex, even though $P(x)$ is not. The  minimization problem that arises in this case is not merely of academic interest -- see for instance \cite{NonConvex} for an iterative scheme where each iteration requires solving such a problem, \cite{nik98ICIP} for a binary denoising formulation which employs weakly penalties. More generally, weakly convex penalties are of interest in sparse signal recovery, because they allow to reduce the bias in the estimates, which arise when convex penalties such as the $\ell_1$ norm are utilized \cite{NonConvex,GNC,cha14ICASSP,bay14ICASSP}. Also, weakly convex penalties are general enough to produce any separable monotone threshold function \cite{ant07p16, bay15p265,cha14ICASSP}.

Through the MM scheme, for $\alpha < 1/\sigma_M$, we obtain in Sec.~\ref{sec:MM} the iterations in \eqref{eqn:ISTA} for a weakly convex penalty and show that the algorithm converges to a minimizer by adapting the proof in \cite{byr14p681}. Specifically, we show the following.
\begin{prop}\label{prop:MMIntro}
Suppose $P$ is proper, lower semi-continuous,  $\rho$-weakly convex with ${0 \leq \rho \leq \sigma_m}$, and the set of minimizers of $C(x)$ is non-empty. If $0 < \alpha < 1/\sigma_M$, then 
the sequence $x^k$ in \eqref{eqn:ISTA} monotonically reduces the cost $C(x)$ and converges to a minimizer of $C(x)$. \qed
\end{prop}

When the penalty is convex, we noted that the  step-size allowed by MM can be doubled. Therefore, one is tempted to ask if larger stepsizes could be used to accelerate the algorithm for when $P$ is weakly convex. The answer is not trivial because, when $\rho >0$, the operator $T_{\alpha}$ loses some of its properties. In particular, it is not firmly non-expansive \cite{Bauschke}, a property which is instrumental in proving Prop.~\ref{prop:FB}. Loss of non-expansivity can be visibly seen by noting that the derivatives of the threshold functions exceed unity when $\rho >0$ (see Figs.~\ref{fig:DemoTHold},~\ref{fig:QuantaPenalty} in this paper or Fig.~2 in \cite{NonConvex}). On the other hand, weak convexity is a mild departure from convexity and one expects some generalization of Prop.~\ref{prop:FB} to  hold. We have the following result  in this direction.
\begin{prop}\label{prop:FB2}
Suppose $P$ is lower semi-continuous, $\rho$-weakly convex with ${0 \leq \rho \leq \sigma_m}$ and the set of minimizers of $C(x)$ is non-empty. If $0<\alpha < 2/(\sigma_M + \rho)$, then  
the sequence $x^k$ in \eqref{eqn:ISTA} converges to a minimizer of $C(x)$. \qed\end{prop}
If we regard $\rho$ as a measure of the deviation of $P$ from being convex, then we see that this deviation from convexity shows itself in the maximum step-size allowed. Note that since $\rho \leq \sigma_m \leq \sigma_M$, the maximum step-size $\alpha$ allowed by Prop.~\ref{prop:FB2} is in general greater than that allowed by Prop~\ref{prop:MMIntro}. We also note that as in the convex case, it can be shown that the cost monotonically decreases with each iteration but this property is not used in the proof of the proposition.

\subsection*{Generalization to an Arbitrary Data Term}
Application of ISTA is not restricted to cost functions that employ quadratic data terms. More generally, consider a generic cost function of the form
\begin{equation}\label{eqn:D}
D(x) = f(x) + P(x),
\end{equation}
where $f:\mathbb{R}^n \to \mathbb{R}$ is a differentiable function. To minimize $D$, ISTA constructs a sequence as
\begin{equation}\label{eqn:ISTAgen}
x^{k+1} = T_{\alpha} \Bigl( x^k - \alpha\,\nabla f(x^k) \Bigr).
\end{equation}
In this setup, the following proposition applies for when $f$ and $P$ are convex \cite{com05p168,Bauschke}.
\begin{prop}\cite{com05p168,Bauschke}\label{prop:FBgeneral}
Suppose $f(x)$ and $P(x)$ are proper, lower semi-continuous, convex and for $\sigma >0$,
\begin{equation*}
\|\nabla f(x) - \nabla f(y) \|_2 \leq \sigma \|x -y \|_2 \text{ for all }x,y.
\end{equation*}
Suppose also that the set of minimizers of $D(x)$ in \eqref{eqn:D} is non-empty.
 If $0< \alpha < 2/\sigma$, then the iterates $x^k$ in \eqref{eqn:ISTAgen} converge to a  minimizer of $D(x)$. \qed
\end{prop}

We note that this is a generalization of Prop.~\ref{prop:FB} since for $f(x) = \frac{1}{2} \|y - H x\|_2^2$, we have $\nabla f (x) =  H^T\,(Hx - y)$, and thus $\|\nabla f (x) - \nabla f (y) \| \leq \sigma_M \|x - y \|_2$. Also, when ${1/ \sigma < \alpha < 2/\sigma}$,\,\,  even though the MM interpretation is not valid, it is possible to show that the cost decreases monotonically \cite{she09p384}.

For weakly convex $P$, we show in this paper that Prop.~\ref{prop:FBgeneral} generalizes as follows.
\begin{prop}\label{prop:FBgeneralNC}
Suppose $f(x)$, $P(x)$ are proper, lower semi-continuos, $P(x)$ is $\rho$-weakly convex, ${f(x) -\frac{\rho}{2}\| x \|_2^2}$ is convex and for $\sigma > 0$,
\begin{equation*}
\|\nabla f (x) - \nabla f (y) \|_2 \leq \sigma \|x - y\|_2 \text{ for all }x,y.
\end{equation*}
Suppose also that the set of minimizers of $D(x)$ in \eqref{eqn:D} is non-empty.
If $0 < \alpha < 2/(\sigma + \rho)$, then the iterates $x^k$ in \eqref{eqn:ISTAgen} converges to a minimizer of $D(x)$. \qed
\end{prop}

As in the case with a quadratic data fidelity term, it can also be shown \cite{kow14ICASSP,she09p384} that the cost decreases monotonically with each iteration. Although the proof of Prop.~\ref{prop:FBgeneralNC} does not depend on this descent property,  we will provide a proof of the following proposition for the sake of completeness. We note, however, that the proof we present follows \cite{kow14ICASSP,she09p384}.

\begin{prop}\label{prop:FBgeneralDescent}
Suppose the hypotheses of Prop.~\ref{prop:FBgeneralNC} hold. Then, $D(x^{k+1}) \leq D(x^k)$.
 \qed
\end{prop}

Although Prop.~\ref{prop:FB2} is a corollary of Prop.~\ref{prop:FBgeneralNC}, we provide an independent proof for  Prop.~\ref{prop:FB2} because showing convergence in this special case allows a more elementary proof. Especially, if $\rho < \sigma_m$, rather than just $\rho \leq \sigma_m$, yet a simpler proof is valid. We will present this simple proof before considering the more general case $\rho \leq \sigma_m$. We also note that the analysis in this paper can be put in a more compact form by resorting to results from non-smooth (non-necessarily convex) analysis \cite{RockafellarWets,Clarke}. However, the additional technical requirements in non-smooth non-convex analysis can be avoided because we are not interested in an arbitrary non-convex problem. Rather, we work under a weak-convexity assumption and this in turn allows us to derive the results using convex analysis methods, simplifying the discussions.

\subsection*{Related Work and Contribution}
Convergence of ISTA with shrinkage/threshold functions other than the soft-threshold has been studied previously in \cite{kow14ICASSP,she09p384,vor13ICASSP}.  The results of  \cite{kow14ICASSP,she09p384} imply that descent in the cost is achieved for the sequence $x^k$, in the sense $C(x^{k+1}) \leq C(x^k)$, when the step-size satisfies $\alpha < 2/(\sigma + \rho)$ as in Prop.~\ref{prop:FBgeneralNC}. This property, along with the global convergence theorem (see Thm.~7.2.3 in \cite{Bazaraa}) implies that accumulation points of $x^k$ minimize the cost function. ISTA with firm-thresholding instead of soft-thresholding is studied in \cite{vor13ICASSP} and it is shown by an MM argument that the cost is reduced at each iteration for $\alpha < 1/ \sigma$. The authors then conclude that the algorithm is convergent by an argument based on the properties of the penalty function. 

In this paper, we present two independent proofs of convergence for ISTA with weakly-convex penalties. The first proof (Prop.~\ref{prop:MMIntro}) uses the MM and SUMMA interpretations of the algorithm and relies on the monotonic descent of the cost with each iteration. This proof is obtained by adapting the proof of convergence presented in \cite{ByrneItOpt,byr14p681} (where it is assumed that the penalty is convex, unlike the case considered here). Unfortunately, in the context of ISTA, this approach does not extend to the case where the algorithm falls out of the MM framework. Our second proof of convergence (for Prop.~\ref{prop:FB2} and Prop.~\ref{prop:FBgeneralNC}), which covers cases that fall out of the MM framework, does not rely on such a descent property. 
We instead study the mapping that ISTA employs, following the schema of \cite{com05p168}.  For weakly convex penalties, such a study has not appeared in the literature as far as we are aware.

\subsection*{Outline}

In Section~\ref{sec:pre}, we recall some definitions and results from convex analysis.
We provide an MM derivation of ISTA and prove Prop.~\ref{prop:MMIntro} in Section~\ref{sec:MM}. To address the cases that fall out of the MM framework, we study in Section~\ref{sec:Operator} the operator that maps $x^k$ to $x^{k+1}$ in \eqref{eqn:ISTA}. Specifically, we show the relation between the fixed points of this operator and the minimizers of the cost in Section~\ref{sec:fixedpt}, study the threshold operator for a weakly convex penalty in Section~\ref{sec:THold}, provide a short and simple convergence proof for the case where the cost is strictly convex in Section~\ref{sec:Strict} and finally provide the proof of Prop.~\ref{prop:FB2} in Section~\ref{sec:avg}. The proofs of Prop.~\ref{prop:FBgeneralNC} and Prop.~\ref{prop:FBgeneralDescent} are given in Section~\ref{sec:GeneralDataFid}. In order to demonstrate that larger stepsizes may be more favorable, we also present two experiments in Section~\ref{sec:experiment}. We conclude with a brief outlook in Section~\ref{sec:conc}.

Some of the frequently used terms/symbols are listed in Table~\ref{table:symb}. We also note that our discussion is restricted to functions defined on $\mathbb{R}^n$ for simplicity.

\begin{table}
\renewcommand{\arraystretch}{1.3}
\centering
\caption{Frequently Used Terms/Symbols \label{table:symb}}
\begin{tabular}{l}
\hline 
$P$ :  penalty function \\
$C$ : cost function from \eqref{eqn:C} \\
$T_\alpha$ :  threshold function defined in \eqref{eqn:T} 
\\
$\alpha$ : step-size for ISTA \\
$\sigma_m, \sigma_M$ :  least and greatest eigenvalues of $H^T\,H$ \\
$\rho$ : weak-convexity parameter of $P$ (see Defn.\ref{defn:weaklyconvex})\\
\hline
\end{tabular}
\end{table}

\section{Definitions and Results from Convex Analysis}\label{sec:pre}

For later reference, we briefly recall some definitions and results from convex analysis in this section. We refer to \cite{Rockafellar,HiriartFund,via83p231} for further discussion. 
\begin{defn}
Suppose $f:\mathbb{R}^n \to \mathbb{R}$ is convex. The subdifferential of $f$ at $x \in \mathbb{R}^n$ is denoted by $\partial f (x)$ and is defined to be the set of $z \in \mathbb{R}^n$ that satisfy
\begin{equation*}
f(x) + \langle y - x, z \rangle \leq f(y), \text{ for all }y. 
\end{equation*}
Any element of $\partial f (x)$ is said to be a subgradient of $f$ at $x$. \qed
\end{defn}
Using the notion of a subdifferential, the minimizer of a convex function can be easily characterized.
\begin{prop}
Suppose $f:\mathbb{R}^n \to \mathbb{R}$ is convex. $x$ minimizes $f$ if and only if $0 \in \partial f(x)$.
\qed
\end{prop}
In order to counter the concavity introduced by a weakly convex penalty function, we will need the data fidelity term to be strongly convex \cite{via83p231}\footnote{Our $\rho$-strongly convex functions correspond to $\rho/2$-convex functions of \cite{via83p231}. }.
\begin{defn}
For $\rho \geq 0$, a function $f:\mathbb{R}^n \rightarrow \mathbb{R}$ is said to be $\rho$-strongly convex if 
\begin{equation*}
h(x) = f(x) - \frac{s}{2} \, \|x\|_2^2
\end{equation*}
is convex when $s \leq \rho$. \qed
\end{defn}
The following lemma, which is of interest in the proximal algorithm, will be used in the sequel. The lemma follows by the optimality conditions.
\begin{lem}\label{lem:prox}
If $h:\mathbb{R}^n \to \mathbb{R} $ is convex and there exists some $x$  and some $\beta > 0$ such that 
\begin{equation*}
h(x) \leq \beta \|z-x\|_2^2 + h(z) \text{ for all } z,
\end{equation*}
then $h$ achieves its minimum at $x$. \qed
\end{lem}

\section{Convergence of ISTA  via Majorization Minimization and SUMMA (Proof of Prop.~\ref{prop:MMIntro})}\label{sec:MM}
In this section, we briefly recall the MM scheme \cite{fig07p980,hun04p30} to show that ISTA achieves monotone descent for a weakly-convex penalty provided that the step size is small enough. We remark however that achieving descent does not automatically imply that the algorithm converges \cite{mey76p108,jac07p411}. In order to show convergence to a minimizer, we follow the approach presented in \cite{byr14p681}, based on  sequential unconstrained minimization algorithms (SUMMA) \cite{ByrneItOpt}. 

\subsection{Descent Property via MM}\label{sec:DescentMM}
Suppose that at the $k\thh$ iteration, we have the estimate $x^k$. 
Let us define,
\begin{align}
g(x,x^k) &= \frac{1}{2} \left\langle x -  x^k, \left(\alpha^{-1} I - H^T\,H\right)\,(x -  x^k)\right\rangle, \label{eqn:gk}\\
M(x,x^k) &= C(x) + g(x,x^k). \label{eqn:M}
\end{align}
Observe that if $\sigma_M< 1/\alpha$, then the matrix $\alpha^{-1} I - H^T\,H$ is positive definite.
Therefore  $g(x,x^k) \geq  0$. It also follows from the definition in \eqref{eqn:gk} that $g(x^k,x^k) = 0$. These observations imply that
\begin{enumerate}[(i)]
\item $M(x,x^k) \geq C(x)$,
\item $M(x^k,x^k) = C(x^k)$.
\end{enumerate}
Thus, starting from $x^k$, we can achieve descent in $C(x)$ by minimizing $M(x,x^k)$ with respect to $x$. After rearranging, $M(x,x^k)$ can be written as,
\begin{equation}\label{eqn:tholdmm}
M(x,x^k) 
=\frac{1}{2\alpha}\,\left\| x -  z^k \right\|_2^2  + P(x)  + \text{const.}
\end{equation}
where $z^k = x^k - \alpha\, H^T(Hx^k - y)$ and `const.' is independent of $x$.
Now since
\begin{equation*}
\rho \leq \sigma_m \leq \sigma_M < 1/ \alpha,
\end{equation*}
the function in \eqref{eqn:tholdmm} is strictly convex with respect to $x$ (for fixed $x^k$) and its unique minimizer is $T_{\alpha}(z^k)$.

From the foregoing discussion, we conclude that if ${\sigma_M< 1/ \alpha}$ and 
\begin{equation}\label{eqn:MM}
x^{k+1} = T_{\alpha}\Bigl( x^k + \alpha\,H^T\,\bigl(y-H\,x^k \bigr)  \Bigr),
\end{equation}
then $C(x^k)$  is a non-increasing sequence. However, this observation alone does not directly imply that $x^k$\, converges to a minimizer. We need a more elaborate argument for proving Prop.~\ref{prop:MMIntro}. 

\subsection{Convergence to a Minimizer via SUMMA}
We start with a key observation.
\begin{lem}\label{lem:M}
For  $g$, $M$, $x^k$\, defined as in \eqref{eqn:gk}, \eqref{eqn:M}, \eqref{eqn:MM}, if ${\sigma_M < 1/\alpha}$, then for any $x$,
\begin{subequations}\label{eqn:ineqs}
\begin{align}
M(x,x^k) - M(x^{k+1},x^k) &\geq \left( \frac{1}{2\alpha} - \frac{\rho}{2} \right)\,\|x - x^{k+1}\|_2^2 \label{ineq1}\\
 & \geq g(x,x^{k+1}).\label{ineq2}
\end{align}
\end{subequations}
\end{lem}

For the proof of this lemma, we use the following auxiliary result, which we will also refer to later.
\begin{lem}\label{lem:aux}
For a given $z$, let $\hat{x} = T_{\alpha}(z)$, where $P$\, is $\rho$-weakly convex. Then, for any $x$,
\begin{equation}\label{eqn:ineqP}
P(x) - P(\hat{x})  \geq -\left\{ \frac{\rho}{2}\|x  - \hat{x}\|_2^2 + \frac{1}{\alpha}\langle \hat{x} - z, x - \hat{x} \rangle \right\}.
\end{equation}
\begin{proof}
Let $P_{\rho}(x) = P(x) + \frac{\rho}{2} \|x\|_2^2$. Note that $P_{\rho}$ is convex. From the definition of $T_{\alpha}$ in \eqref{eqn:T} we have,
\begin{align*}
\hat{x} &= \arg \min_x \frac{1}{2\alpha} \left\| x  - z\right \|_2^2  + P(x) \\
& = \arg \min_x \left(\frac{1}{2\alpha} - \frac{\rho}{2} \right) \| x  - z \|_2^2  - \rho \langle x,z\rangle +  P_{\rho}(x)
\end{align*}
By the optimality conditions, we obtain
\begin{equation*}
\left(\frac{1}{\alpha} - \rho \right)(z- \hat{x}) + \rho z  \in   \partial P_{\rho}(\hat{x}).
\end{equation*}
But by the definition of subdifferential, this implies that
\begin{equation*}
P_{\rho}(x) \geq P_{\rho}(\hat{x})  + \left \langle \left(\alpha^{-1} - \rho \right)(z- \hat{x}) + \rho z,\, x - \hat{x} \right \rangle.
\end{equation*}
Plugging in the definition of $P_{\rho}$ and rearranging, we obtain \eqref{eqn:ineqP}.
\end{proof}
\end{lem}

Using this lemma, we obtain the proof of Lemma~\ref{lem:M} as follows.
\begin{proof}[Proof of Lemma~\ref{lem:M}]
Let $f(x) = \frac{1}{2}\|y - Hx\|_2^2$. Observe that $\nabla f(x) = H^T\,(Hx - y)$.
Using \eqref{eqn:tholdmm}, and Lemma~\ref{lem:aux} with ${x^{k+1} = T_{\alpha}\bigl(x^k - \alpha \nabla f(x^k)\bigr)}$, we obtain,
\begin{align*}
M&(x,x^k) - M(x^{k+1},x^k)  \\
&= \left\{ \frac{1}{2\alpha} \Bigl( \|x - x^k\|_2^2 - \|x^{k+1} - x^k\|_2^2 \Bigr) \right\}  \\ &\quad + \bigl\langle x - x^{k+1} , \nabla f(x^k)\bigr\rangle
+\Bigl\{ P(x) - P(x^{k+1}) \Bigr\}   \\
&\geq \left\{ \frac{1}{2\alpha} \|x - x^{k+1}\|_2^2 + \frac{1}{\alpha} \langle x - x^{k+1},x^{k+1} - x^k\rangle \right\}   \\
& \quad + \langle x - x^{k+1} , \nabla f(x^k)\rangle  \\
& \quad - \left\{ \frac{\rho}{2}\|x  - x^{k+1} \|_2^2 \right. \\
& \quad  \left. +\frac{1}{\alpha}\langle x^{k+1} - x^k + \alpha \nabla f(x^k), x - x^{k+1} \rangle \right\}\\
 &= \left( \frac{1}{2\alpha} - \frac{\rho}{2} \right)\,\|x - x^{k+1}\|_2^2.
\end{align*}
To see \eqref{ineq2}, note that $H^T\,H - \rho I \geq 0$ and observe 
\begin{align*}
\frac{1}{\alpha} I - H^T\,H &= \left(\frac{1}{\alpha} - \rho \right)\,I - \bigl(H^T\,H - \rho I \bigr)\\
&\leq  \left(\frac{1}{\alpha} - \rho \right)\,I.
\end{align*}
\end{proof}

We are now ready to present the proof of convergence. Although there are some variations, the main idea follows the proof of Theorem~4.1 in \cite{byr14p681}.
\begin{proof}[Proof of convergence for Prop.~\ref{prop:MMIntro}]
By assumption, $C$ is proper and its set of minimizers is non-empty. These imply that ${\inf_x C(x) = c > -\infty}$. Combining this the discussion in Sec.~\ref{sec:DescentMM}, we deduce that $C(x^k)$\, is a non-increasing sequence which is bounded from below and therefore it converges to some $b\geq c$.

Since the set of minimizers of $C(x)$ is non-empty by assumption, we can find $\hat{x}$\, such that $C(\hat{x})\leq b$. Consider ${d_k(\hat{x}) =  M(\hat{x},x^k) - M(x^{k+1},x^k)}$. We know by the definition of $x^{k+1}$ that $d_k(\hat{x})$ is non-negative. We now show that it is also non-increasing with $k$. Note that by Lemma~\ref{lem:M}, we have $g(\hat{x},x^k) \leq d_{k-1}(\hat{x})$. Using this, we obtain
\begin{align*}
d_k(\hat{x}) &=M(\hat{x},x^k) - M(x^{k+1},x^k)\\
&= g(\hat{x},x^k) + C(\hat{x})  - g(x^{k+1},x^k) - C(x^{k+1}) \\
&\leq d_{k-1}(\hat{x}) + \Bigl\{ C(\hat{x})  - g(x^{k+1},x^k) - C(x^{k+1}) \Bigr\}  \\
& \leq d_{k-1}(\hat{x}).
\end{align*}
where the last line follows because $C(\hat{x})\leq C(x^k)$ for all $k$\, and $g(x^{k+1},x^k)$ is non-negative.

Again by Lemma~\ref{lem:M}, we can now conclude that
\begin{equation*}
\| \hat{x} - x^k \|_2^2 \leq \left( \frac{1}{2\alpha} - \frac{\rho}{2} \right)^{-1}\,\Bigl( M(\hat{x},x^0) - M(x^{1},x^0) \Bigr),
\end{equation*}
for all $k$. Thus $x^k$\, is a bounded sequence and it has a convergent subsequence by the Bolzano-Weierstrass theorem \cite{Rudin}. Let $x^*$\, be the limit of a convergent subsequence $x^{k_n}$. Observe that since $C(x^k)$ is non-increasing and $C(\cdot)$ is lower semi-continuous, we have $C(x^*) \leq  b$. Also, for a given $\epsilon>0$, we can find some $N$ such that if $n \geq N$, then 
\begin{equation*}
\left( \frac{1}{2\alpha} - \frac{\rho}{2} \right) \|x^* - x^{k_n}\|_2^2 \leq  \epsilon.
\end{equation*}
It thus follows by \eqref{ineq2}, along with $C(x^*) \leq C(x^{k_n+1})$ and the non-negativity of $g$ that 
\begin{align*}
& M(x^*, x^{k_n}) - M(x^{k_n+1},x^{k_n}) \\
 & = C(x^*) +  g(x^*,x^{k_n}) - C(x^{k_n+1}) -  g(x^{k_n+1},x^{k_n}) \\
 & \leq g(x^*,x^{k_n}) \\
  &\leq \epsilon.
\end{align*} 
 But since $\bigl[M(x^*, x^{k}) - M(x^{k+1},x^{k})\bigr]$ is a non-increasing, non-negative sequence of real numbers, this implies that the whole sequence converges to zero. In view of \eqref{ineq1}, we thus conclude that
\begin{equation*}
\lim_{k\to \infty}  \|x^* - x^{k+1}\|_2^2 = 0.
\end{equation*}
Therefore, the whole sequence converges to $x^*$. 

What remains is to show that $x^*$\, is actually a global minimizer of $C$. For that, it suffices to show $b = c$. 
Let $z$\, be a minimizer of $C$, i.e., $C(z) = c$. From Lemma~\ref{lem:M}, we obtain
\begin{align*}
C(z) + g(z,x^k) - &C(x^{k+1}) - g(x^{k+1},x^k)  \\ &= M(z,x^k) - M(x^{k+1},x^k)  \\
 & \geq g(z,x^{k+1}).
\end{align*}
Rearranging, we have,
\begin{align*}
  g(z,x^k) - g(z,x^{k+1}) &\geq   C(x^{k+1}) - C(z) + g(x^{k+1},x^k)\\
  & \geq b - c.
\end{align*}
Now if $b > c$, this implies that the sequence $g(z,x^{k})$\, decreases without bound, but this cannot happen since $g(z,x^{k}) \geq 0$ for all $k$. Therefore we must have $b \leq c$. But we already know that $c \leq b$, so it must be $b = c$.
\end{proof}

\section{ISTA as Fixed-Point Iterations of an Operator}\label{sec:Operator}
When $P(x)$ is convex, the forward-backward splitting algorithm \cite{com05p168,Bauschke} leads to iterations that are of the same form as \eqref{eqn:MM}. However, the results of \cite{com05p168} imply that the maximum step size allowed by MM can in fact be doubled while still ensuring convergence. This in turn accelerates convergence significantly. 
We now investigate this issue for a weakly-convex penalty $P(x)$.

In order to simplify our analysis, we decompose the operator in \eqref{eqn:MM}. We define
\begin{equation}\label{eqn:U}
U_{\alpha}(x) =  \alpha\,H^T\,y + \bigl( I  - \alpha\, H^T\,H)\,x,
\end{equation}
and rewrite the iterations in \eqref{eqn:MM} as,
\begin{equation*}
x^{k+1} = T_{\alpha}\Bigl( U_{\alpha}\bigl(x^k \bigr) \Bigr).
\end{equation*}

In the following, we will first study the fixed points of the composite operator $T_{\alpha}\,U_{\alpha}$ and show an equivalence with the minimizers of the cost $C(x)$. Then, we study the properties of the two operators $T_{\alpha}$ and $U_{\alpha}$. Under a strict convexity assumption, we will see that the composition is actually a contraction mapping. If we lift the strictness restriction from the convexity assumption, the composite operator turns out to be averaged (see Sec.~\ref{sec:avg}).

\subsection{Fixed Points of the Algorithm}\label{sec:fixedpt}
We now establish a relation between the fixed points of $T_{\alpha}\,U_{\alpha}$ and the minima of $C(x)$. Specifically, our goal in this subsection is to show the following result.
\begin{prop}\label{prop:fixed}
Suppose $P(x)$ is $\rho$-weakly convex, $T_{\alpha}$ is as defined in \eqref{eqn:T} and $\alpha\,\rho < 1$. Then, $x = T_{\alpha}\Bigl( U_{\alpha}(x) \Bigr)$ if and only if $x$ minimizes $C(x)$ in \eqref{eqn:C}.
\end{prop}

Instead of proving Prop.~\ref{prop:fixed} directly, we will prove a more general form using convex analysis methods. This general form will also be referred to in the proof of Prop.~\ref{prop:FBgeneralNC} in Sec.~\ref{sec:GeneralDataFid}.
\begin{prop}\label{prop:fixedGeneral}
Suppose $P:\mathbb{R}^n \to \mathbb{R}$ is a $\rho$-weakly convex function, and $f:\mathbb{R}^n \to \mathbb{R}$ is a  differentiable, $\rho$-strongly convex function. Suppose also that $\alpha\,\rho <1$. Under these conditions, 
\begin{equation}\label{eqn:fixed}
x = T_{\alpha}\,\Bigl( x- \alpha\,\nabla f(x) \Bigr),
\end{equation}
if and only if $x$ minimizes $f+P$.
\begin{proof}
($\Rightarrow$) Suppose \eqref{eqn:fixed} holds. We will show that $x$ minimizes $f + P$.
Let $u = x - \alpha\,\nabla f (x)$. By the definition of $T_{\alpha}$, $x = T_{\alpha}(u)$ means that
\begin{equation*}
\frac{1}{2}\,\|x-  u \|_2^2 + \alpha\,P(x) \leq \frac{1}{2}\,\|z -  u \|_2^2 + \alpha\,P(z), \text{ for all } z.
\end{equation*}
Noting that $x - u = \alpha\,\nabla f (x)$ and $z-u = (z-x) + \alpha\,\nabla f(x)$, we can rewrite this as
\begin{multline*}
\frac{1}{2}\,\|\alpha\,\nabla f(x) \|_2^2 + \alpha\,P(x) \\ \leq \frac{1}{2}\,\|z -  x \|_2^2 + \frac{1}{2}\,\|\alpha\,\nabla f(x) \|_2^2 \\ + \langle z-x, \alpha\,\nabla f(x) \rangle + \alpha\,P(z), \text{ for all } z.
\end{multline*}
Cancelling  $\frac{1}{2}\,\|\alpha\,\nabla f(x) \|_2^2$ from both sides and noting that \\ $\langle z-x, \alpha\,\nabla f(x) \rangle \leq \alpha\,f(z) - \alpha\,f(x)$ (because $f$ is convex), we obtain
\begin{equation*}
\alpha\,P(x) \leq \frac{1}{2}\,\|z -  x \|_2^2 + \alpha\,f(z) - \alpha\,f(x) + \alpha\,P(z), \text{ for all } z.
\end{equation*}
Rearranging,
\begin{equation*}
f(x) + P(x) \leq \frac{1}{2\,\alpha}\,\|z -  x \|_2^2 + f(z) + P(z), \text{ for all } z.
\end{equation*}
Since $f + P$ is convex, by Lemma~\ref{lem:prox}, we conclude that $x$ minimizes $f+P$.

($\Leftarrow$) Assume that $x$ minimizes $h = f+P$. Let $P_{\rho}(t) =  P(t) + (\rho/2)\,\|t\|_2^2$ and $f_{\rho}(t) = f(t) - (\rho/2)\,\|t\|_2^2$. Note that by assumption $f_{\rho}$ and $P_{\rho}$ are both convex and $h = f_{\rho} + P_{\rho}$. Since $x$ minimizes $h$, we have,
\begin{equation*}
0 \in  \underbrace{\{\nabla f(x) - \rho\,x\}}_ {\partial f_{\rho}(x)} + \partial P_{\rho}(x)
\end{equation*}
or,
\begin{equation*}
\rho\,x - \nabla f(x)\in \partial P_{\rho}(x)
\end{equation*}
Adding any multiple of $x$ to both sides, we find that for $s \geq \rho$, and $P_{s}(t) = P(t) + (s/2)\,\|t\|_2^2$, we have
\begin{equation*}
s\,x - \nabla f(x)\in \partial P_s(x)
\end{equation*}
By the convexity of $P_s$ we then obtain,
\begin{equation*}
P_s(x) + \langle z-x, s\,x - \nabla f(x) \rangle \leq P_s(z), \text{ for all } z.
\end{equation*}
Rearranging, we have
\begin{multline*}
\frac{s}{2}\|x\|_2^2 - \langle x, s\,x - \nabla f(x) \rangle + P(x) \\ \leq \frac{s}{2}\|z\|_2^2 - \langle z, s\,x - \nabla f(x) \rangle + P(z), \text{ for all } z.
\end{multline*}
Equivalently, for $s\geq \rho$,
\begin{multline*}
\frac{1}{2}\Bigl\|x - \bigl(x - s^{-1}\nabla f(x)\bigr) \Bigr\|_2^2 + \frac{1}{s}\,P(x) \\ \leq \frac{1}{2}\Bigl\|z - \bigl(x - s^{-1}\nabla f(x)\bigr) \Bigr\|_2^2 + \frac{1}{s}\,P(z), \text{ for all } z.
\end{multline*}
Now let $\alpha = 1/s$. The inequality above may be written as,
\begin{multline}\label{eqn:tempin1}
\frac{1}{2}\Bigl\|x - \bigl(x - \alpha\,\nabla f(x)\bigr) \Bigr\|_2^2 + \alpha\,P(x) \\ \leq \frac{1}{2}\Bigl\|z - \bigl(x - \alpha\,\nabla f(x)\bigr) \Bigr\|_2^2 + \alpha\,P(z), \text{ for all } z,
\end{multline}
for $\alpha\,\rho \leq 1$.
Note that equality in \eqref{eqn:tempin1} may be achieved by setting $z = x$. But we know that for $\alpha \,\rho < 1$, the right hand side is uniquely minimized by $z = T_{\alpha}\bigl(x - \alpha\,\nabla f(x)\bigr)$. Thus, $x = T_{\alpha}\bigl(x - \alpha\,\nabla f(x)\bigr)$ for $\alpha\,\rho < 1$.
\end{proof}
\end{prop}
Prop.~\ref{prop:fixed} is a corollary of this proposition. This can be seen by taking $f(x) = \frac{1}{2}\,\|y - H x\|_2^2$ and noting that it is $\rho$-strongly convex since $\rho \leq \sigma_m$.

\subsection{Threshold Operators Associated with  Weakly Convex Penalties}\label{sec:THold}

We now study the operator $T_{\alpha}$. We only assume that $T_{\alpha}$ is associated with a $\rho$-weakly convex penalty $P$ via \eqref{eqn:T}. 

\begin{defn}
An operator $S : \mathbb{R}^n \to \mathbb{R}^n$ is said to be non-expansive if,
\begin{equation*}
\|S(x) - S(z) \|_2 \leq  \|x -z \|_2.
\end{equation*}
\end{defn}
We will make use of the following result (see \cite{combettes_chp, Bauschke} for instance).
\begin{lem}\label{lem:nonexp}
Suppose $q:\mathbb{R}^n\to\mathbb{R}$ is convex and the operator $S(x)$ is defined as
\begin{equation*}
S(x) = \arg \min_t \,\frac{1}{2}\,\|x - t\|_2^2 + q(t).
\end{equation*}
Then,
\begin{equation*}
\| S(x) - S(z) \|_2 \leq \|x -z \|_2. 
\end{equation*} 
\qed
\end{lem}
This can be shown by making use of the optimality conditions and monotonicity of the subgradient.
\begin{prop}\label{prop:T}
Suppose $P(x)$ is $\rho$-weakly convex and ${\alpha \,\rho < 1}$. Then,
\begin{equation}\label{eqn:propT}
\| T_{\alpha}(x) - T_{\alpha}(z) \|_2 \leq \frac{1}{1-\alpha\,\rho} \| x - z \|_2.
\end{equation}
\begin{proof}
Note that the function
\begin{equation*}
\frac{\rho}{2} \|x\|_2^2 + P(x)
\end{equation*}
is convex. Therefore,
\begin{equation*}
S(x) = \arg \min_t \frac{1}{2} \|x - t\|_2^2 + c\,\Bigl( \frac{\rho}{2} \|t\|_2^2 + P(t) \Bigr)
\end{equation*}
is non-expansive by Lemma~\ref{lem:nonexp} for any $c>0$. But we have
\begin{align*}
S(x) &= \arg \min_t \frac{1}{2} \|x - t\|_2^2 + c\,\Bigl( \frac{\rho}{2} \|t\|_2^2 + P(t) \Bigr)\\
&= \arg \min_t \frac{1+c\rho}{2} \|t\|_2^2 -  \langle x,t\rangle + c\,P(t)\\
&= \arg \min_t \frac{1}{2} \left\|t - \frac{1}{1+c\rho} x \right\|_2^2 + \frac{c}{1+c\rho}\,P(t) \\
&= T_{c\,(1+c\rho)^{-1}} \left( \frac{1}{1+c\rho} x \right)
\end{align*}
Therefore we deduce that
\begin{multline}\label{eqn:tmp1}
\left\| T_{c\,(1+c\rho)^{-1}} \left( \frac{1}{1+c\rho} x \right) - T_{c\,(1+c\rho)^{-1}} \left( \frac{1}{1+c\rho} z \right) \right\|_2 \\ = \|S(x) - S(z)\|_2 \leq  \| x - z\|_2.
\end{multline}
Now set $\alpha = c\,(1+c\rho)^{-1}$. Observe that $\alpha\,\rho < 1$ for any $c>0$. Solving for $c$, we have $c = \alpha\,(1-\alpha\,\rho)^{-1}$. Plugging this in \eqref{eqn:tmp1}, we obtain
\begin{equation*}
\left\| T_{\alpha} \Bigl( (1-\alpha\rho) x \Bigr) - T_{\alpha} \Bigl( (1-\alpha\rho) x \Bigr) \right\|_2 \leq \| x - z\|_2.
\end{equation*}
Making a change of variables, we finally obtain \eqref{eqn:propT}.
\end{proof}
\end{prop}

\subsection{ISTA as Iterations of a Contraction Mapping}\label{sec:Strict}

In this section, we derive a convergence result that is relatively easy to obtain, under the additional restriction $\rho < \sigma_m$. We desire not to exclude the case $\rho = \sigma_m$, because for $\rho = 0$ (a convex penalty function), we would like to allow $\sigma_m = 0$, which corresponds to an operator $H$ with a non-trivial null-space. In Sec.~\ref{sec:avg}, we will also allow the case $\rho = \sigma_m$, leading to a generalization of Prop.~\ref{prop:FB}.

\begin{prop} Suppose that the eigenvalues of $H^T\,H$ are contained in the interval $[\sigma_m, \sigma_M]$, $P(x)$ is $\rho$-weakly convex and $T_{\alpha}(\cdot)$ is as given in \eqref{eqn:T}.
If
\begin{subequations}\label{eqn:hypo}
\begin{align}
\rho &< \sigma_m, \label{eqn:hypo1} \\
\alpha &< \frac{2}{\sigma_M + \rho},\label{eqn:hypo2}
\end{align}
\end{subequations}
then, the iterations in \eqref{eqn:MM} converge to the unique minimizer of $C(x)$.
\begin{proof}
For $U_{\alpha}$ in \eqref{eqn:U}, we have that 
\begin{equation*}
\| U_{\alpha}(x) - U_{\alpha}(z) \|_2 \leq \max\bigl(|1 - \alpha\,\sigma_M|, |1 - \alpha\,\sigma_m| \bigr)\,\|x - z\|_2.
\end{equation*}
Now observe that \eqref{eqn:hypo1} implies $\rho < (\sigma_M + \rho)/2$. This along with \eqref{eqn:hypo2}, gives $\alpha\,\rho <1$. 
Therefore, by Prop.~\ref{prop:T}, we can write
\begin{multline*}
\Bigl\|T_{\alpha}\,\bigl( U_{\alpha}(x)\bigr) - T_{\alpha}\,\bigl( U_{\alpha}(z)\bigr)  \Bigr\|_2 \\ \leq \frac{\max\bigl(|1 - \alpha\,\sigma_M|, |1 - \alpha\,\sigma_m| \bigr)}{1 - \alpha\,\rho}\,\|x - z\|_2,
\end{multline*}
When $\rho < \sigma_m$, $C(x)$ is strictly convex and the minimizer, which exists by assumption, is unique. By Prop.~\ref{prop:fixed}, this unique minimizer is in fact the unique fixed point of $T_{\alpha}\,U_{\alpha}$. Let us denote this minimizer as $z$. Notice that
\begin{multline*}
\|T_{\alpha}\bigl( U_{\alpha}\,(x)\bigr) - z \|_2 \\ \leq \frac{\max\bigl(|1 - \alpha\,\sigma_M|, |1 - \alpha\,\sigma_m| \bigr)}{1 - \alpha\,\rho}\,\|x - z\|_2.
\end{multline*}
Thus the iterations converge (geometrically) to $z$ if the two conditions below hold
\begin{subequations}\label{eqn:cond}
\begin{align}
\frac{|1 - \alpha\,\sigma_M|}{1 - \alpha\,\rho} &< 1, \label{eqn:cond1} \\
\frac{|1 - \alpha\,\sigma_m|}{1 - \alpha\,\rho} &< 1\label{eqn:cond2}.
\end{align}
\end{subequations}
\eqref{eqn:cond1} is equivalent to 
\begin{subequations}\label{eqn:cond1eq}
\begin{align}
\rho &< \sigma_M,\\
\alpha &< \frac{2}{\sigma_M + \rho}.
\end{align}
\end{subequations}
\eqref{eqn:cond2} is equivalent to 
\begin{subequations}\label{eqn:cond2eq}
\begin{align}
\rho &< \sigma_m,\\
\alpha &< \frac{2}{\sigma_m + \rho}.
\end{align}
\end{subequations}
Noting that $\sigma_M \geq \sigma_m$, and $2/(\sigma_M + \rho) \leq 2/(\sigma_m+\rho)$ we deduce that \eqref{eqn:hypo} implies \eqref{eqn:cond1eq}, and \eqref{eqn:cond2eq}, completing the proof.
\end{proof}
\end{prop}

\begin{remark}
Observe that since $\rho < \sigma_m \leq \sigma_M$, we have
\begin{equation*}
\frac{2}{\sigma_M + \rho} > \frac{1}{\sigma_M}.
\end{equation*}
Thus the generalized forward backward algorithm converges for stepsizes greater than that allowed by majorization-minimization. We also note that the convergence proof given for Prop.~\ref{prop:MMIntro} does not straightforwardly extend to the case considered in this subsection.
\end{remark}

\subsection{ISTA as Iterations of an Averaged Operator}\label{sec:avg}

\begin{defn}\cite{Bauschke}
An operator $S:\mathbb{R}^n \to \mathbb{R}^n$ is said to be $\beta$-averaged with $\beta \in (0,1)$ if $S$ can be written as 
\begin{equation*}
S = (1-\beta)\,I + \beta\,U,
\end{equation*}
for a non-expansive $U$. \qed
\end{defn}

As a corollary of the definition, we have,
\begin{lem}\label{lem:avg}
$S:\mathbb{R}^n \to \mathbb{R}^n$ is $\beta$-averaged if and only if 
\begin{equation*}
\frac{1}{\beta}\Bigl(S - (1-\beta)\,I\Bigr)
\end{equation*}
is non-expansive. \qed
\end{lem}

Averaged operators are of interest because they behave more desirably concerning convergence. Further, as we will note, averaged-ness is preserved under composition, which is instrumental in proving the convergence of the forward backward splitting algorithm \cite{Bauschke,com05p168}. To demonstrate the difference between the behavior of an averaged operator and a non-expansive operator, let us consider a scenario as follows. Let $U$ be non-expansive with a unique fixed point $z$, and  $x$ an arbitrary point. Then, $\|U(x) - z \|_2 \leq \|x - z \|_2$, but $U(x)$ is not guaranteed to be closer to $z$ than $x$. In the worst case, both $x$ and $U(x)$ might be equidistant to $z$. This is illustrated in Fig.~\ref{fig:averaged}. Now let $S = (1- \beta) I + \beta\,U$  for $\beta \in (0,1)$. $z$ is also the unique fixed point of $S$. But now, since $S(x)$ lies somewhere on the open segment between $x$ and $U(x)$, we will have $\|S(x) - z\|_2 < \| x - z \|_2$ (see Fig.~\ref{fig:averaged}). This discussion is of course not a proof of convergence for iterated applications of $S$ but it provides some intuition. The following proposition, which is also known as the Krasnosels'ki\u{\i}-Mann theorem (see Thm.~5.14 in \cite{Bauschke}) provides a convergence result for averaged operators.

\begin{prop}\label{prop:AvgConv}
Suppose $S:\mathbb{R}^n \to \mathbb{R}^n$ is $\beta$-averaged and its set of fixed points is non-empty. Given $x^0$, let $x^{n+1} = S (x^n)$. Then the sequence $\{x^n\}_{n\in \Z}$ converges and the limit is a fixed point of $S$. \qed
\end{prop}

To prove the convergence of ISTA, we will show that $T_{\alpha} \bigl( U_{\alpha}(\cdot)\bigr)$ is an averaged operator. In order to show that this operator is averaged, we will show that  $T_{\alpha}$ and $U_{\alpha}$ can be regarded as averaged operators by proper scaling (see $S_{\alpha}$ and $V_{\alpha}$\, below) and invoke the following result. 

\begin{figure}
\centering
\includegraphics[scale=1]{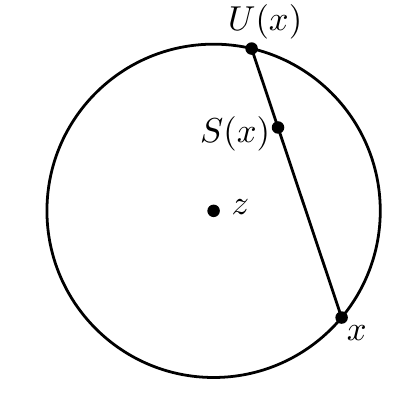}
\caption{ A non-expansive operator $U$ might not always take a point $x$ closer to its fixed point $z$. However, an averaged operator $S = (1- \beta) I + \beta\,U$  derived from $U$ will have such a property. \label{fig:averaged}}
\end{figure}

\begin{prop}\label{prop:comp}
If $S_1$ and $S_2$ are averaged, then $S_1\bigl( S_2(\cdot) \bigr)$  is also averaged. \qed
\end{prop}

The proof of this proposition follows from the definition of an averaged operator 
and is omitted (for an alternative statement and proof of this proposition, see Prop.~4.32 in \cite{Bauschke}).

We will also need the following result (see Cor.23.8 in \cite{Bauschke}).
\begin{lem}\label{lem:1/2}
Suppose $q:\mathbb{R}^n \to \mathbb{R}$ is convex and $S(x)$ is defined as 
\begin{equation*}
S(x) = \arg \min_t \,\frac{1}{2}\,\|x - t\|_2^2 + q(t).
\end{equation*}
Then, $S$ is 1/2-averaged. \qed
\end{lem}
Note that this is a stronger result than Lemma~\ref{lem:nonexp}. In fact, Lemma~\ref{lem:nonexp} follows as a corollary of Lemma~\ref{lem:1/2}.

Using Lemma~\ref{lem:1/2}, we can show that the composition of the threshold operator $T_{\alpha}$ with a proper scaling is averaged.
\begin{prop}\label{prop:fn}
For $\alpha\,\rho < 1$, let $S_{\alpha}$ be defined as
\begin{equation}\label{eqn:scaledT}
S_{\alpha}(x) = T_{\alpha}\Bigl( (1-\alpha\,\rho) x \Bigr).
\end{equation} 
Then, $S_{\alpha}$ is $1/2$-averaged.
\begin{proof}
In the proof of Prop.~\ref{prop:T},  for $\alpha\,\rho < 1$, we show that $S_{\alpha}$ can actually be expressed as
\begin{equation*}
S_{\alpha}(x) = \arg \min_t \frac{1}{2} \|x - t\|_2^2 + \frac{\alpha}{1-\alpha\,\rho}\,\Bigl( \rho \|t\|_2^2 + P(t) \Bigr).
\end{equation*}
Since $\rho \|t\|_2^2 + P(t)$ is convex, this means that $S_{\alpha}$ is  1/2-averaged by Lemma~\ref{lem:1/2}.
\end{proof}
\end{prop}

For affine operators that employ a symmetric matrix, such as those used in ISTA with a quadratic data term, averaged-ness is associated with the eigenvalues of the matrix.
\begin{prop}\label{prop:matrix}
Suppose $S:\mathbb{R}^n\rightarrow \mathbb{R}^n$ is an affine mapping of the form
\begin{equation*}
S(x) = M\,x + u,
\end{equation*}
where $M$ is a symmetric matrix, i.e. $M = M^T$, and $u$ is a constant vector. Then, 
$S$ is $\beta$-averaged for some $\beta \in (0,1)$ if and only if the eigenvalues of $M$ lie in the interval $(-1,1]$.
\begin{proof}
Suppose the eigenvalues of $M$ lie in the interval $[\sigma_0, \sigma_1]$.
Consider 
\begin{equation*}
V = \frac{1}{\beta}\Bigl(S - (1-\beta)\,I\Bigr).
\end{equation*}
Then,
\begin{equation*}
V(x) - V(z) = \underbrace{ \frac{1}{\beta}\, \Bigl(M - (1-\beta)I \Bigr) }_Q (x-z).
\end{equation*}
The eigenvalues of $Q$ lie in the interval,
\begin{equation*}
\left[ \frac{\sigma_0 - 1 + \beta }{\beta}, \frac{\sigma_1 - 1 + \beta }{\beta}\right].
\end{equation*}
Therefore, by Lemma~\ref{lem:avg}, $S$ is $\beta$-averaged for some $\beta \in (0,1)$ iff $Q$ is non-expansive. But $Q$ is non-expansive for some ${\beta \in (0,1)}$ iff 
\begin{multline}\label{eqn:eqcond}
-1 \leq \frac{\sigma_0 - 1 + \beta }{\beta} \leq \frac{\sigma_1 - 1 + \beta }{\beta} \leq 1, \\ \text{ for some }\beta \in (0,1).
\end{multline}
 Note that since $\beta$ is restricted to be positive, \eqref{eqn:eqcond} can be written equivalently as,
\begin{equation}\label{eqn:eqcond2}
1-2\beta \leq \sigma_0  \leq \sigma_1  \leq 1, \text{ for some }\beta \in (0,1).
\end{equation}
But \eqref{eqn:eqcond2} holds iff $-1< \sigma_0 \leq \sigma_1 \leq 1$.
\end{proof}
\end{prop}
In the convergence proof, to counter the scale factor that appears in \eqref{eqn:scaledT}, we use the following result.
\begin{prop}\label{prop:avg}
For $\alpha\,\rho < 1$, let $V_{\alpha}$ be defined as
\begin{equation*}
V_{\alpha} = \frac{1}{1-\alpha\,\rho}U_{\alpha}.
\end{equation*} 
Then, $V_{\alpha}$ is averaged if the two conditions below are satisfied
\begin{subequations}
\begin{align}
\rho &\leq \sigma_m, \label{eqn:h1}\\
\alpha &< \frac{2}{\sigma_M + \rho}.\label{eqn:h2}
\end{align}
\end{subequations}
\begin{proof}
Observe that $V_{\alpha}$ is of the form
\begin{equation*}
V_{\alpha}(x) = \underbrace{\frac{1}{1-\alpha\,\rho} \, \bigl( I  - \alpha\, H^T\,H)}_M\,x + u,
\end{equation*}
for a constant vector $u$. The eigenvalues of $M$ are contained in the interval
\begin{equation*}
\left[ \frac{1 - \alpha\,\sigma_M}{1-\alpha\,\rho}, \frac{1 - \alpha\,\sigma_m }{1-\alpha\,\rho}\right].
\end{equation*}
By \eqref{eqn:h1}, we have
\begin{equation*}
\frac{1 - \alpha\,\sigma_m}{1-\alpha\,\rho} \leq 1.
\end{equation*}
By \eqref{eqn:h2}, we have
\begin{equation*}
\frac{1 - \alpha\,\sigma_M}{1-\alpha\,\rho} > -1.
\end{equation*}
Thus, it follows by Prop.~\ref{prop:matrix} that $V_{\alpha}$ is averaged.
\end{proof}
\end{prop}

We are now ready for the proof of Prop.~\ref{prop:FB2}.
\begin{proof}[Proof of Prop.~\ref{prop:FB2}]
Note that $T_{\alpha}\,U_{\alpha} = S_{\alpha}\,V_{\alpha}$, where $S_{\alpha}$ and $V_{\alpha}$ are as defined in Prop.~\ref{prop:fn} and Prop.~\ref{prop:avg}. By Prop.~\ref{prop:fn} and Prop.~\ref{prop:avg}, $S_{\alpha}$ and $V_{\alpha}$ are averaged. Then, by Prop.~\ref{prop:comp} we conclude that $T_{\alpha}\,U_{\alpha}$ is also averaged. We also have by Prop.~\ref{prop:fixed} that the fixed points of $T_{\alpha}\,U_{\alpha}$ coincide with the set of minimizers of $C(x)$, which is non-empty by assumption. The claim now follows by Prop.~\ref{prop:AvgConv}.
\end{proof}

\section{General Data Fidelity Term}\label{sec:GeneralDataFid}
In this section, we provide proofs  of Prop.~\ref{prop:FBgeneralNC} and Prop.~\ref{prop:FBgeneralDescent}. 
Let us recall the setup. We consider a cost function of the form
\begin{equation*}
D(x) = f(x) + P(x),
\end{equation*}
where $P$ is $\rho$-weakly convex and $f$ is $\rho$-strongly convex, differentiable with
\begin{equation}\label{eqn:Lipschitz}
\| \nabla f(x)-  \nabla f(y)\| \leq \tau \|x -y \|, \text{ for all }x,y.
\end{equation}
\begin{defn}
A function $f:\mathbb{R}^n \to \mathbb{R}$ is said to have a $\tau$-Lipschitz continuous gradient, if \eqref{eqn:Lipschitz} holds. \qed
\end{defn}

Let us now recall the Baillon-Haddad theorem \cite{Bauschke,bau10p781,ByrneBaillon}.
\begin{lem}
Suppose $f$ is convex, differentiable and 
it has a $\tau$-Lipschitz continuous gradient. Then,
\begin{equation*}
\langle \nabla f(x) - \nabla f(z), x - z \rangle \geq \frac{1}{\tau}\,  \| \nabla f(x) - \nabla f(z)\|_2^2.
\end{equation*}
\end{lem}

In the setting above, we have,
\begin{lem}\label{lem:g}
Suppose $f$ is $\rho$-strongly convex, differentiable and its gradient is $\sigma$-Lipschitz continuous with $\sigma > \rho$. Also, let $g(x) = f(x) - \rho \, \| x\|_2^2/2$. Then, $\nabla g$ is $(\sigma - \rho)$-Lipschitz continuous.
\begin{proof}
Let $F = \nabla f$ and $G = \nabla g$. Note that $G = F - \rho I$.
First observe that, by the Baillon-Haddad theorem applied to $F$, we have,
\begin{equation*}
\langle F(x) - F(y) , x-y \rangle \geq \frac{1}{\sigma} \| F(x)-  F(y) \|_2^2.
\end{equation*}
Now,
\begin{align*}
\| &G(x)-  G(y) \|_2^2 = \| F(x) - F(y) \|_2^2 + \rho^2 \|x - y \|_2^2 \\ & \qquad \qquad \qquad \qquad - 2\rho \langle F(x) - F(y) , x-y \rangle  \\
&\leq \left( 1 - 2\frac{\rho}{\sigma} \right) \|F(x) - F(y)\|_2^2 + \rho^2 \|x - y \|_2^2 \\
&\leq \left( 1 - 2\frac{\rho}{\sigma} \right) \sigma^2\,\|x - y\|_2^2 + \rho^2 \|x - y \|_2^2\\
&= (\sigma - \rho)^2\,\|x - y \|_2^2.
\end{align*}
Taking square roots, we obtain 
\begin{equation*}
\|G(x) - G(y)\|_2 \leq (\sigma-\rho)\|x-  y\|_2,
\end{equation*}
 which is the claim.
\end{proof}
\end{lem}

\begin{prop}\label{prop:gradient}
Suppose $f$ is $\rho$-strongly convex, differentiable and its gradient is $\sigma$-Lipschitz continuous with $\sigma > \rho$. Also, for $\alpha\,\rho < 1$, let $V_{\alpha}$ be defined as
\begin{equation*}
V_{\alpha} = \frac{1}{1-\alpha\,\rho}\,\Bigl( I - \alpha \nabla f  \Bigr).
\end{equation*} 
Then, $V_{\alpha}$ is averaged if 
\begin{equation}\label{eqn:condF}
\alpha < \frac{2}{\sigma + \rho}.
\end{equation}
\begin{proof}
Let $g(x) = f(x) - \rho \, \| x\|_2^2/2$. Then, $g$ is convex, $\nabla g = \nabla f - \rho I$ and $\nabla g$ is $(\sigma - \rho)$-Lipschitz. By the Baillon-Haddad theorem, we have
\begin{equation}\label{ineq:g}
\langle \nabla g(x) - \nabla g(y), x-y \rangle \geq \frac{1}{\sigma - \rho}\, \|\nabla g(x) - \nabla g(y) \|_2^2.
\end{equation}
Also,
\begin{align*}
V_{\alpha} &= \frac{1}{1-\alpha\,\rho}\,\Bigl( I - \alpha \bigl(\rho I + \nabla g \bigr)  \Bigr) \\
&= I - \frac{\alpha}{1-\alpha\,\rho} \nabla g.
\end{align*}
Assuming \eqref{eqn:condF} holds, set $\beta = \alpha (\sigma + \rho)/2$ and observe that $0 < \beta < 1$. Observe also that
\begin{equation*}
Q = \frac{1}{\beta}\,\bigl(V_{\alpha} - (1- \beta) I \bigr)= I - c\,\nabla g,
\end{equation*}
where $c = \dfrac{2}{(1-\alpha\,\rho) (\sigma + \rho)}$. We now have, by \eqref{ineq:g},
\begin{multline*}
\|Q(x) - Q(y)\|_2^2 \leq \|x - y \|_2^2 \\  + \left(c^2 - \frac{2c}{\sigma - \rho} \right) \|\nabla g(x) - \nabla g(y)\|_2^2.
\end{multline*}
It can be checked that \eqref{eqn:condF} implies
\begin{equation*}
\left(c^2 - \frac{2c}{\sigma - \rho} \right) < 0.
\end{equation*}
Therefore $Q$\, is non-expansive and by Lemma~\ref{lem:avg}, $V_{\alpha}$ is $\beta$-averaged.
\end{proof}
\end{prop}

We now present the proof of Prop.~\ref{prop:FBgeneralNC}. The argument is similar to that in the proof of Prop.~\ref{prop:FB2}.
\begin{proof}[Proof of Prop.~\ref{prop:FBgeneralNC}]
Let us define the operator $A$ through
\begin{equation*}
A\, x = T_{\alpha} \Bigl( x - \alpha\,\nabla f(x) \Bigr),
\end{equation*}
Then, for $\alpha\,\rho < 1$, $A$ can also be written as $A = S_{\alpha}\,V_{\alpha}$
 where $S_{\alpha}$ and $V_{\alpha}$ are as defined in Prop.~\ref{prop:fn} and Prop.~\ref{prop:gradient}. But if $\alpha < 2/(\sigma + \rho)$, $S_{\alpha}$ and $V_{\alpha}$ are averaged by Prop.~\ref{prop:fn} and Prop.~\ref{prop:gradient}. Then, by Prop.~\ref{prop:comp}, we conclude that $A$ is also averaged. We also have by Prop.~\ref{prop:fixedGeneral} that the fixed points of $A$ comprise the set of minimizers of the convex function $D(x)$ which is non-empty by assumption. The claim now follows by Prop.~\ref{prop:AvgConv}.
\end{proof}

We remark that the convergence proof for ISTA did not use and does not directly imply that the cost decreases monotonically with each iteration (as stated in Prop.~\ref{prop:FBgeneralDescent}). However, such a descent property has been shown to hold in \cite{she09p384,kow14ICASSP} previously. We include a proof of this descent property (i.e., Prop.~\ref{prop:FBgeneralDescent}) for the sake of completeness. The proof depends on two lemmas.
\begin{lem}\label{lem1}
The sequence constructed in \eqref{eqn:ISTAgen} satisfies,
\begin{multline*}
P(x^{k+1}) -  P(x^k) + \langle \nabla f (x^k),  x^{k+1} - x^k \rangle \\ \leq  \left( \frac{\rho}{2} -  \frac{1}{\alpha}\right) \|x^k - x^{k+1}\|_2^2.
\end{multline*}
\begin{proof}
This claim follows from Lemma~\ref{lem:aux} by plugging $x^k$\, for $x$, $x^{k+1}$ for $\hat{x}$ and ${x^k - \alpha \nabla f(x^k)}$ for $z$ (since ${x^{k+1} = T_{\alpha}(x^k - \alpha\,\nabla f(x^k))}$ ).
\end{proof}
\end{lem}

\begin{lem}\label{lem2}
Suppose $h:\mathbb{R}^n \to \mathbb{R}$ is a continuous differentiable function with a $\tau$-Lipschitz continuous gradient.
Then,
\begin{equation*}
h(x) \leq h(y) + \langle \nabla h(y), (x-y) \rangle + \frac{\tau}{2} \,\|x -y \|_2^2.
\end{equation*}
\end{lem}
For a proof of this lemma, see 3.2.12 in \cite{Ortega}.

We remark  that for the data fidelity term $f$, the Lipschitz constant is $\sigma$.

\begin{proof}[Proof of Prop.~\ref{prop:FBgeneralDescent}]
By Lemma~\ref{lem2}, we have,
\begin{equation*}
f(x^{k+1}) \leq f(x^k) + \langle \nabla f(x^k), x^{k+1}-x^k \rangle + \frac{\sigma}{2} \,\|x^{k+1} -x^k \|_2^2.
\end{equation*}
Adding $P(x^{k+1})$ to both sides, we can rewrite this as,
\begin{align*}
D(x^{k+1}) \leq & D(x^k) +  \Bigl\{ P(x^{k+1}) - P(x^k) \Bigr. \\  + & \Bigl. \langle \nabla f(x^k), x^{k+1}-x^k \rangle + \frac{\sigma}{2} \,\|x^{k+1} -x^k \|_2^2  \Bigr\} 
\end{align*}
Let us denote the term in curly brackets with $c$. If we can show that $c$ is non-positive, we are done. By Lemma~\ref{lem1}, we have,
\begin{equation*}
c \leq \left( \frac{\rho}{2} - \frac{1}{\alpha} + \frac{\sigma}{2} \right)\, \|x^{k+1} -x^k \|_2^2.
\end{equation*}
But since $\alpha < 2/ (\sigma + \rho)$, the term in the parentheses is negative and thus $c < 0$.
\end{proof}

\section{Experiments}\label{sec:experiment}
In this section, we present two different experiments to evaluate the acceleration achieved by increasing the step size. We also compare the speed of convergence with those achieved by acceleration methods like  TwIST \cite{bio07p992} and FISTA \cite{bec09p183}. We note however, that TwIST and FISTA are proposed for when convex penalties are used in the cost function. Therefore, their convergence analyses are not valid in the current setup. Nevertheless, we have observed that they lead to considerable acceleration and included them to give an idea about what can be gained by increasing the step size. We also note that our purpose in these experiments is not to promote weakly convex penalties but to demonstrate the effects of increasing the step size in ISTA. Matlab code for the experiments is available at ``\texttt{\small http://web.itu.edu.tr/ibayram/NCISTA/}''.

\begin{experiment}\label{exp:sparse}
Our first experiment involves a sparse signal recovery problem.
We use a $60\times 50$ convolution matrix $H$  associated with an invertible filter (so that $H^T\,H$ is invertible) to construct the observed signal. Using a sparse $x$ as shown in Fig.\ref{fig:demo1}a, we produce the observed signal as 
\begin{equation*}
y = H\,x + u,
\end{equation*}
where $u$ denotes white Gaussian noise. We use a  penalty function $P:\mathbb{R} \rightarrow \mathbb{R}$ defined as
\begin{equation*}
P_{\tau,\rho}(s) = \begin{cases}
\tau\,|s| - s^2/(2\rho), &\text{ if } |s| < \tau/\rho,\\
\tau^2/(2\rho), &\text{ if } |s| \geq \tau/\rho,\\
\end{cases}
\end{equation*}
This function is shown in  Fig.~\ref{fig:DemoTHold}a.
This penalty function is $\rho$-weakly convex. The threshold function associated with $P$ (see \eqref{eqn:T}) is known as a firm-threshold \cite{gao97p855} and is given by (provided $\alpha\,\rho < 1$),
\begin{equation*}
T_{\alpha}(s) = \begin{cases}
0, &\text{if } |s| < \alpha\,\tau,\\
(1-\alpha\,\rho)^{-1}\,(s - \alpha\,\tau), &\text{if } \alpha\,\tau \leq |s| < \tau/\rho,\\
s, &\text{if } \tau/\rho \leq |s|.
\end{cases}
\end{equation*}
This threshold function is depicted in Fig.~\ref{fig:DemoTHold}b. Observe that for $\alpha \,\tau < |s| < \tau/\rho$, the derivative of the firm-threshold function exceeds unity. Therefore the firm-threshold function is not non-expansive.

\begin{figure}
\centering
\centering
\includegraphics[scale=1]{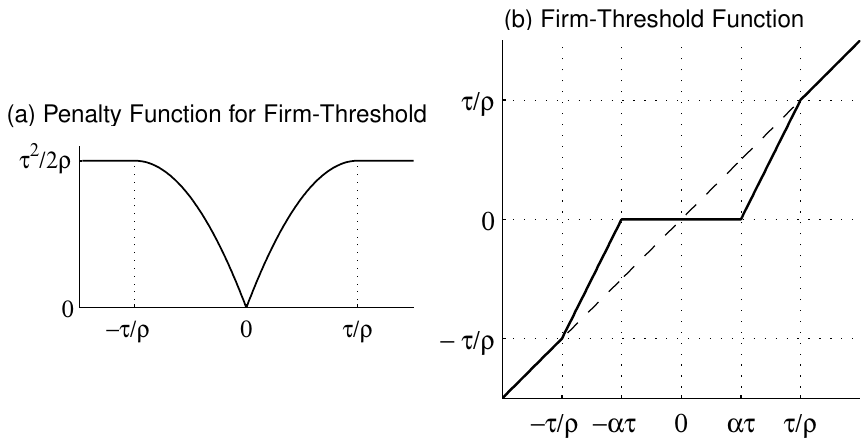}
\caption{The penalty and threshold function used in Exp.~\ref{exp:sparse}. Observe that this threshold function is not non-expansive. Specifically, the derivative exceeds unity on the intervals $(-\tau/\rho, -\alpha\rho)$ and $(\alpha\rho, \tau/\rho)$ \label{fig:DemoTHold}}
\end{figure}

In the setup described above, we set $\rho$ as the least eigenvalue of $H^T\,H$, which is the maximum value allowed if a convex cost is desired. We also set $\tau = 3\,\rho \,\std(u)$, where $\std(u)$ denotes the standard deviation of noise. We obtain the estimate of the sparse signal as,
\begin{equation}\label{eqn:ExpCost}
x^* = \arg \min_t \frac{1}{2} \|y - H\,t\|_2^2 + \sum_i P_{\tau,\rho}(t_i).
\end{equation}
Note that $P$ applies componentwise to an input vector $t$. The threshold function $T_{\alpha}$ is also applied componentwise in the realization of ISTA.

Denoting the greatest eigenvalue of $H^T\,H$ by $\sigma$, we set $\alpha_0 = 1/\sigma$ and $\alpha_1 = 2/(\sigma+\rho)$. Note that $\alpha_0$ is the greatest value of the stepsize allowed by MM (see Prop.~\ref{prop:MMIntro}), whereas for $\alpha_1$, convergence is guaranteed by Prop.~\ref{prop:FB2}. For this specific problem, the ratio $\alpha_1/\alpha_0$ was found to be 1.88. 

\begin{figure}
\centering
\includegraphics[scale=1]{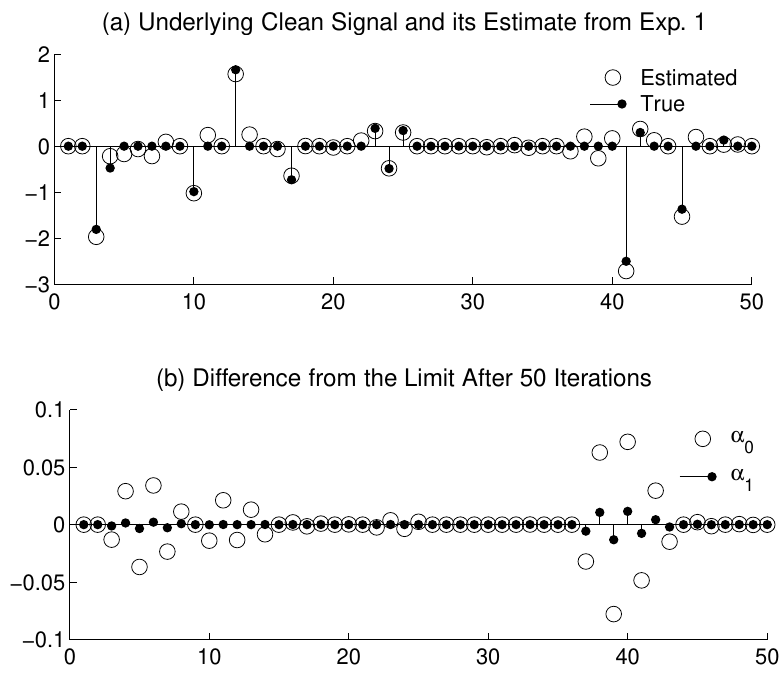}
\caption{(a) The true signal and its estimate obtained by solving \eqref{eqn:ExpCost}. (b) The difference from the minimizer after 50 iterations of ISTA with stepsizes $\alpha_0 < \alpha_1$. \label{fig:demo1}}
\end{figure}

\begin{figure}[h!]
\centering
\includegraphics[scale=1]{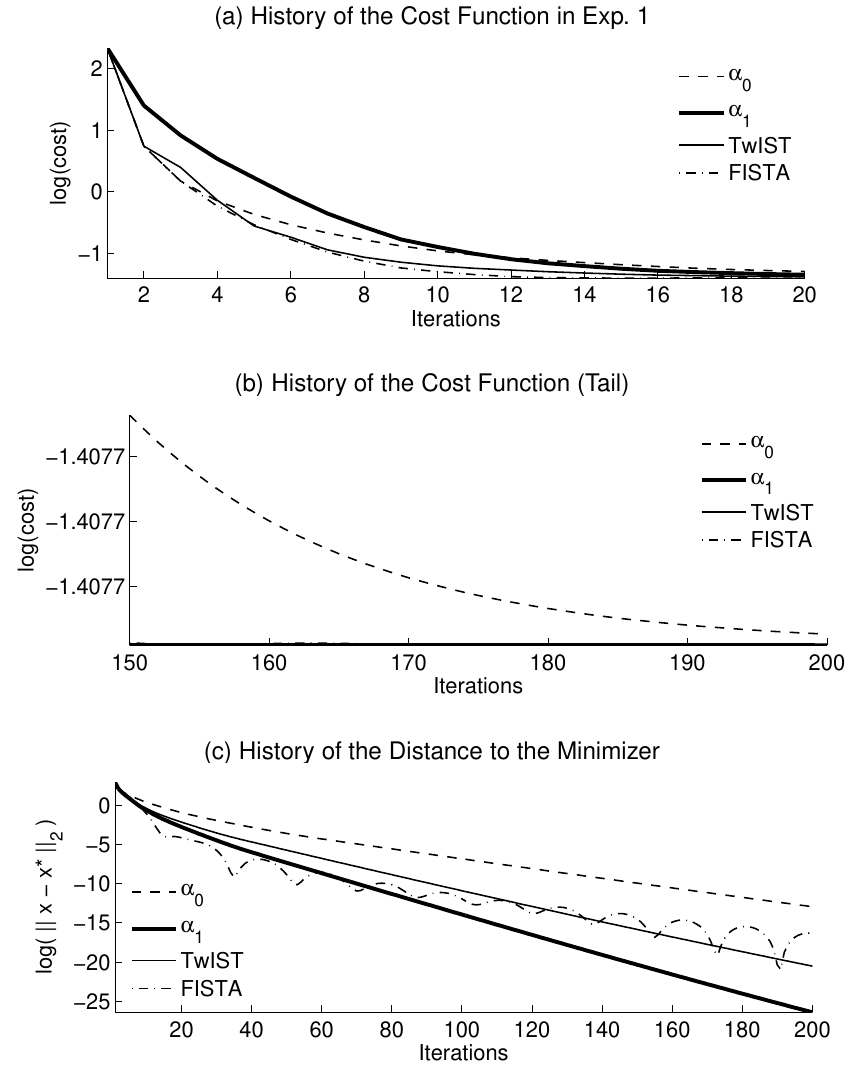}
\caption{ Effects of changing the step size $\alpha$ in ISTA for Exp.~\ref{exp:sparse}.  (a,b) show how the cost function evolves with iterations using the greatest possible $\alpha_0$ that guarantees strict descent and the greatest possible $\alpha_1$ that merely guarantees convergence. The cost with TwIST and FISTA are also included.  (c) shows the distance to the unique minimizer with respect to iterations for the same four choices.  \label{fig:demo1,2}}
\end{figure}

We ran ISTA with $\alpha_0$ and $\alpha_1$, in both cases, starting from zero. In order to better evaluate the convergence speed, we also tried FISTA \cite{bec09p183} and TwIST \cite{bio07p992}. We note that both methods require that the components of the cost function be convex and thus they were not originally proposed for the current setup. Nevertheless, both methods are formally applicable. We note that at least in the convex case, FISTA requires the step-size to be less than or equal to $1/\sigma$. Therefore for FISTA, we employed the stepsize $\alpha_0 = 1/\sigma$ (we also tried it for $\alpha_1$, but the sequence diverged). Given the step-size, we applied the algorithm referred to as `FISTA with constant stepsize' in Sec.4 of \cite{bec09p183}. For TwIST, there are two parameters to choose, namely `$\alpha$' and `$\beta$' -- for these, we used the suggestions in equations (26) and (27) of \cite{bio07p992}.

The history of the cost function with iterations is shown in Fig.~\ref{fig:demo1,2}~a,b. With $\alpha_0$, initial descent is greater compared to that achieved by $\alpha_1$ but eventually the cost achieved by $\alpha_1$ drops below that achieved by $\alpha_0$. FISTA and TwIST perform quite well in the beginning, dropping the cost faster than ISTA with $\alpha_0$ or $\alpha_1$. However, after about 30 iterations, FISTA, TwIST and ISTA (with $\alpha_1$) achieve almost the same  cost values. 
We found these observations to be quite stable with respect to different noise realizations.

In order to assess the convergence speed, using $\alpha_0$ we ran ISTA for 10K iterations to obtain an estimate of the minimizer $x^*$. Then, we reran the algorithms mentioned above. The logarithm of the Euclidean distance to $x^*$ with respect to iterations is shown in Fig.~\ref{fig:demo1,2}c. We see that the the distance to the minimizer decreases faster when larger steps are used in ISTA. TwIST converges with a rate that lies in between those of ISTA with stepsizes $\alpha_0$ and $\alpha_1$. The behavior of FISTA is less stable. We see bursts that take the iterate close to the limit followed by slight departures. Especially in the first few iterations, the distance to the minimizer is greatly reduced by FISTA but in the long run, convergence rate is slightly worse than that of TwIST, on average. 
To conclude, although TwIST and FISTA are very successful in reducing the cost rapidly, this does not necessarily lead to the fastest rate of convergence when we monitor the distance to the minimizer. 

\end{experiment}

\begin{experiment}\label{exp:discrete}
In this experiment, we consider a clean signal $x(n)$ composed of piecewise constant blocks of three samples. That is, the signal satisfies
\begin{equation}\label{eqn:mod}
x(3n) = x(3n+1) = x(3n+2),
\end{equation}
for every integer $n$. We further know that $x(n)$ takes integer values in the range $[0,4]$. An example is shown in Fig.~\ref{fig:QuantaObs}a (stem plot). We observe a blurred and noisy version of this signal, namely $y$. The least squares reconstruction of  $x$ given the observed $y$ is also shown in Fig.~\ref{fig:QuantaObs}a (circles). In order to take \eqref{eqn:mod} into account, we propose to synthesize $x$ as a weighted linear combination of atoms $\{g(n - 3\,k) \}_{k\in \mathbb{Z}}$, where
$g(n) = \delta(n) + \delta(n-1) + \delta(n-2)$. That is, we let our estimate $\hat{x}$ be given as,
\begin{equation*}
\hat{x} = \sum_k \hat{c}(k)\,g(n-3k),
\end{equation*}
where $\hat{c}$ denotes the coefficients to be estimated. Let us denote the operator that takes $c$ to $\hat{x}$ as $G$. Also, let $F$ denote the blurring operator. In this setting, we wish to estimate $\hat{c}$ as
\begin{equation}\label{eqn:exp2Formulation}
\hat{c} = \arg \min_c \frac{1}{2} \| y - H\,c\|_2^2 + \tau\,\sum_k P\bigl( c(k) \bigr), 
\end{equation}
where $H = F\,G$ and $P:\mathbb{R}\to\mathbb{R}$ is a penalty function that applies to the components of $c$. We expect $P$ to penalize deviations from the integers in the range $[0,4]$. Notice that this is inherently a non-convex constraint. No useful convex penalty function exists for enforcing such a constraint. However, viable weakly-convex penalties can be found for this task (see also \cite{nik98ICIP} in this context). We propose to use a penalty function given as
\begin{equation*}
P(s) = \begin{cases}
\infty, &\text{if }s < 0,\\
(s - \lfloor s \rfloor)\,(\lceil s \rceil - s), & \text{if } 0 \leq s \leq 4,\\
\infty, &\text{if }4 < s,
\end{cases}
\end{equation*}
where $\lfloor\cdot \rfloor$ and $\lceil \cdot \rceil$ denote the floor and ceiling functions respectively.
This penalty function is 2-weakly convex and is shown in Fig.~\ref{fig:QuantaPenalty}a.
The associated threshold function $T_{\alpha}(s)$ for $\alpha < 1/2$ is given as 
\begin{equation*}
\begin{cases}
0, &\text{if } s < 0,\\
\lfloor s \rfloor, &\text{if } 0\leq  s   \leq \lfloor s \rfloor + \alpha \leq 4 + \alpha ,\\
\lfloor s \rfloor + \frac{s - \lfloor s \rfloor - \alpha}{1-2\alpha}, & \text{if }  0\leq \lfloor s \rfloor +\alpha \leq s \leq \lceil s \rceil -\alpha \leq 4,\\
\lceil s \rceil &\text{if }  0\leq \lceil s  \rceil - \alpha \leq s\leq 4,\\
4, &\text{if }  4 < s.
\end{cases}
\end{equation*}
The threshold function is shown in Fig.~\ref{fig:QuantaPenalty}b. Observe that the threshold function allows non-integer values but favors integers (it has integer valued deadzones). Note also that this threshold function applies componentwise.

\begin{figure}
\centering
\includegraphics[scale=1]{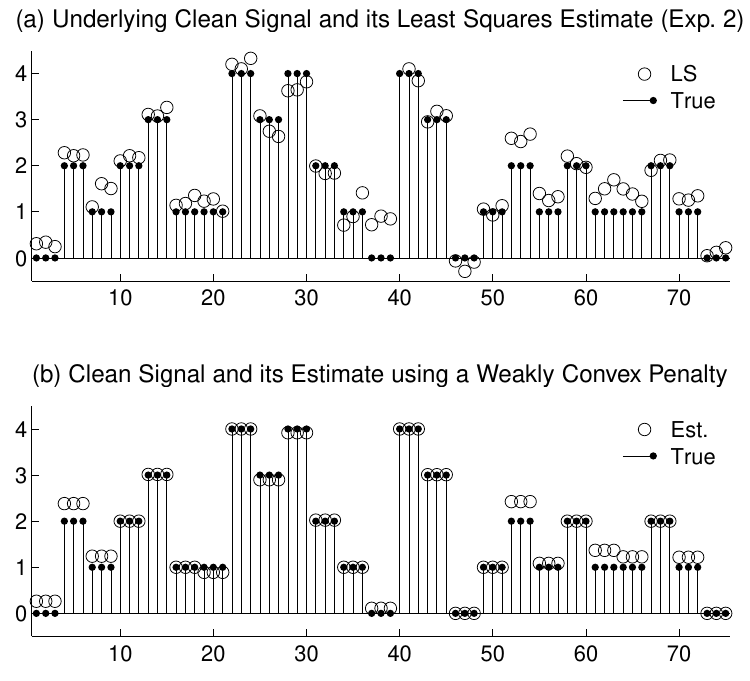}
\caption{(a) The true signal and its least squares estimate for Exp.~\eqref{exp:discrete}. (b) The clean signal and its estimate obtained by solving \eqref{eqn:exp2Formulation}. \label{fig:QuantaObs}}
\end{figure}

\begin{figure}
\centering
\includegraphics[scale=1]{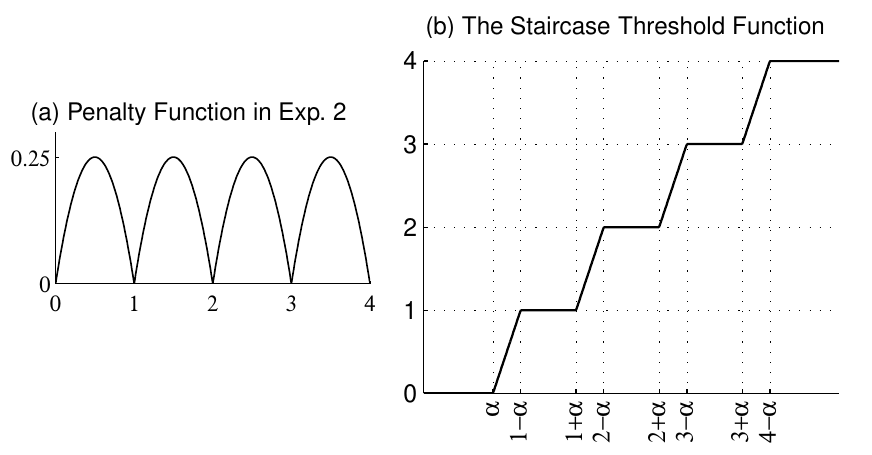}
\caption{(a) The penalty function used in Exp.~\ref{exp:discrete} to promote integer values in the range $[0,4]$. (b) The threshold function $T_{\alpha}$ per \eqref{eqn:T} associated with the function in (a). \label{fig:QuantaPenalty}}
\end{figure}

In this setup, we ran ISTA with the four choices described in Experiment~\ref{exp:sparse}, namely ISTA with $\alpha_0$ (largest step-size allowed by the MM framework), ISTA with $\alpha_1$ (largest step-size for which convergence is ensured), TwIST and FISTA. We chose the parameters for TwIST and FISTA similarly as in  Exp.~\ref{exp:sparse}. The history of the cost function and the distance to the limit is shown in Fig.~\ref{fig:demo2}a,b respectively. The figures are in agreement with those in Exp.~\ref{exp:sparse}. Increased step sizes lead to acceleration that is about the same order (in fact, better in this example) as that achieved by TwIST and FISTA. We again note, however that we do not have a theoretical explanation for this behavior (whereas FISTA is theoretically optimal in terms of reducing the cost for convex penalties).

\begin{figure}
\centering
\includegraphics[scale=1]{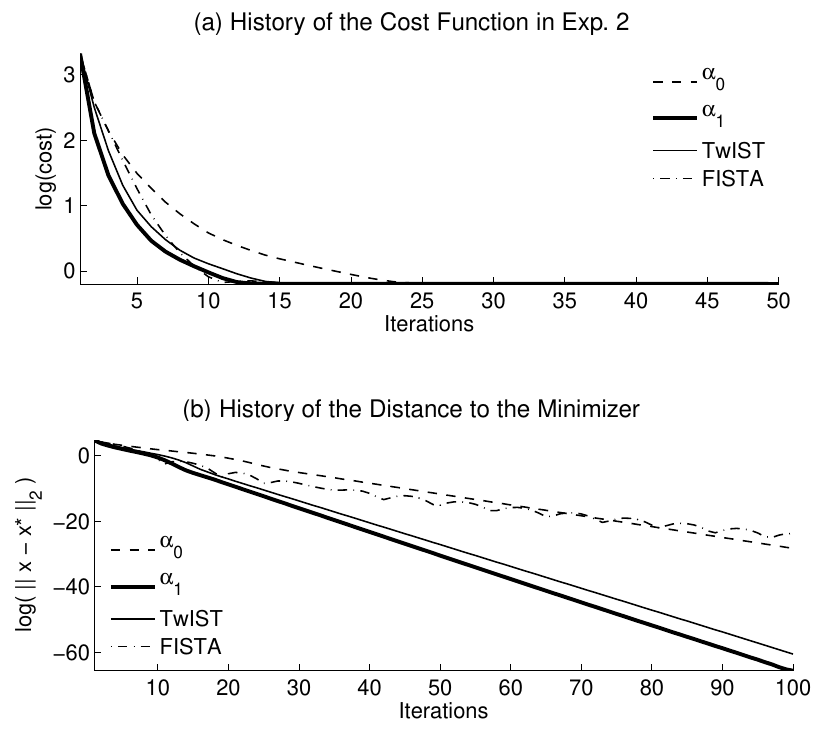}
\caption{ (a) shows the history of the cost function for Exp.~\ref{exp:discrete}. The same choices as in Exp.~\ref{exp:sparse} are used for comparison. (b) shows the history of the distance to the unique minimizer.  \label{fig:demo2}}
\end{figure}

\end{experiment}

\section{Conclusion}\label{sec:conc}

In this paper, we studied the convergence of ISTA when the penalty term is weakly convex and provided a generalized convergence condition. We  also demonstrated that  larger step-sizes lead to faster convergence, although we do not have a precise theoretical justification at the moment. The generalization in this paper relies on a study of the proximity operator for a weakly convex function. Specifically, we have seen that the proximity operator for a weakly convex function is no longer (firmly) non-expansive and/or averaged. However, under proper scaling, the proximity operators regain these properties. Given this observation, it is also natural to consider extensions of this work to other algorithms that employ proximity operators, specifically the Douglas-Rachford algorithm \cite{eck92p293}, which can in turn be used to derive other algorithms such as the alternative direction method of multipliers (ADMM) \cite{eck92p293,boy11p1}, or the parallel proximal algorithm (PPXA) \cite{pus11p450}. We hope to investigate these extensions in future work.

\subsection*{Acknowledgements}
The author thanks the reviewers for their comments and suggestions that helped improve the paper.

\bibliographystyle{IEEEtranS}

\end{document}